\documentclass[12pt,twoside]{amsart}
\usepackage{amssymb,amsmath,amsthm, amscd, enumerate, mathrsfs, upgreek}
\usepackage{graphicx, hhline, tikz}
\usepackage[colorlinks=true,pagebackref,hyperindex]{hyperref}
\usepackage[all]{xy}

\usepackage{url}

\usepackage{color}
\usepackage[backrefs, alphabetic, initials]{amsrefs}
\usepackage{color}  
\usepackage{latexsym, bm}
\usepackage[T1]{fontenc} 
\usepackage{fancyhdr}
\usepackage{enumerate}
\title[Strictly nef anti-canonical divisors]
{On a numerical criterion for Fano fourfolds}

\author{Haidong Liu}

\date{\today, version 0.03}

\address{Sun Yat-sen University, Department of mathematics, Guangzhou, 510275, China}
\email{liuhd35@mail.sysu.edu.cn, jiuguiaqi@gmail.com}

\DeclareMathOperator{\strutt}{\mathcal{O}}

\newtheorem{thm}{Theorem}[section]
\newtheorem{lem}[thm]{Lemma}
\newtheorem{prop}[thm]{Proposition}
\newtheorem{conj}[thm]{Conjecture}
\newtheorem{cor}[thm]{Corollary}
\newtheorem{ques}[thm]{Question}

\theoremstyle{definition}

\newtheorem{defn}[thm]{Definition}
\newtheorem{rem}[thm]{Remark}
\newtheorem*{ack}{Acknowledgments}  

\newtheorem{step}{Step}

\newtheorem*{claim}{Claim}

\makeatletter
    
    \@addtoreset{equation}{section}
\makeatother
\begin{document}

\begin{abstract}
In this paper, we prove a special case of Campana--Peternell's conjecture in dimension 4. Specifically, we show that a projective smooth fourfold $X$ with $c^2_1(X)\cdot c_2(X)\neq 0$ 
and strictly nef anti-canonical divisor $-K_X$ is a Fano fourfold.
To this aim, we completely solve
the non-vanishing conjecture for strictly nef anti-canonical divisors in dimension 4.
\end{abstract}

\maketitle 

\tableofcontents

\section{Introduction}\label{sec1}

In \cite{cp}, Campana and Peternell proposed the following conjecture,
which provides a numerical criterion for Fano manifolds.

\begin{conj}[\cite{cp}*{Problem 11.4}]\label{conj.cp}
Let $X$ be a projective manifold 
such that the anti-canonical divisor $-K_X$ is strictly nef,
i.e., $-K_X\cdot C>0$ for any curve $C$ on $X$.
Then, $X$ is a Fano manifold,  i.e.,  $-K_X$ is ample.
\end{conj}

By \cite{loy}*{Theorem 1.2}, a projective manifold  with  strictly nef anti-canonical divisor
is rationally connected. 
Therefore, Conjecture \ref{conj.cp} is equivalent to

\begin{conj}\label{conj.cp.eq}
Let $X$ be a projective rationally connected manifold 
such that the anti-canonical divisor $-K_X$ is strictly nef.
Then,  $-K_X$ is ample.
\end{conj}

Then, in this paper, Campana--Peternell's conjecture stands for Conjecture \ref{conj.cp} or \ref{conj.cp.eq},
depending on whether the rationally connectedness should be emphasized.
This conjecture has been confirmed in dimension 2 by Maeda \cite{maeda} and in dimension 3
by Serrano \cite{serrano}. In Serrano's work \cite{serrano}, 
Campana--Peternell's conjecture was put into a more general 
framework by viewing as a special case of the so-called Serrano's conjecture,
and then could be treated by induction on the dimension. In this paper, we follow
Serrano's idea (as \cites{ccp, liu-matsumura} did) 
and give an affirmative answer to Campana--Peternell's conjecture in dimension 4 under 
the assumption that $c^2_1(X)\cdot c_2(X)\neq 0$.

\begin{thm}[$=$ Corollary \ref{cor.main} $+$ Theorem \ref{thm.nonvanishing}]\label{thm.main}
Let $X$ be a projective rationally connected smooth fourfold such that $c^2_1(X)\cdot c_2(X)\neq 0$ 
and the anti-canonical divisor $-K_X$ is strictly nef.
Then, $X$ is a Fano fourfold.
\end{thm}

Following closely to \cite{liu-matsumura}, the organization of this paper is divided into two parts.
In Section \ref{sec3}, we partially prove the ampleness part of Theorem \ref{thm.main} under the assumption that  $\kappa(X, -K_X)\geq 0$.
By the same reason that 
Serrano's conjecture for Calabi--Yau threefolds has not been completely solved yet,
we have to get around this difficulty by carefully using the strict nefness as in \cite{liu-matsumura}.
Unfortunately, 
we cannot solve the problem completely at this moment.
In this paper, the case $c^2_1(X)\cdot c_2(X)=0$ still remains.

In Section \ref{sec4}, we completely solve
the non-vanishing conjecture for strictly nef anti-canonical divisors (Conjecture \ref{conj.nonvanishing}) in dimension 4.

\begin{thm}[Theorem \ref{thm.nonvanishing}]
Let $X$ be a projective rationally connected smooth fourfold such that 
the anti-canonical divisor $-K_X$ is strictly nef.
Then, we have that $\kappa(X, -K_X)\geq 0$.
\end{thm}

\begin{ack}
The author would like to thank Chen Jiang, Jie Liu, Wenhao Ou and Guolei Zhong for useful discussions and suggestions.
\end{ack}

Throughout this paper,  we work over the complex number field $\mathbb C$.
A {\em scheme} is always assumed to be separated and of finite type over $\mathbb{C}$, 
and a {\em variety} is a reduced and irreducible algebraic scheme. 
We will freely use the notation 
in \cites{demailly-book, fujino-foundations, kollar-mori, laz}.

\section{Preliminaries}\label{sec2}

In this section, we present some preliminary results.

\subsection{Hirzebruch--Riemann--Roch formula}
Let $X$ be a projective smooth fourfold such that $-K_X$ is nef but not big.
Then, since $c_1(X)=-K_X$ and  $K_X^4=0$, we obtain that
\begin{equation}\label{eq.rr}
\chi(X, -mK_X)= \frac{c^2_1(X)\cdot c_2(X)}{24}(m^2+m)+\chi(X, \mathcal O_X)
\end{equation}
from  the Hirzebruch--Riemann--Roch formula.



\begin{lem}\label{lem.rr}
Let $X$ be a projective rationally connected smooth fourfold 
such that $-K_X$ is nef but not big.
Then, $\chi(X, -mK_X)\geq 1$ for any $m\geq 0$.
\end{lem}

\begin{proof}
Since $c_1(X)=-K_X$ is nef, we obtain that $c^{2}_1(X)\cdot c_2(X)=c_2(X)\cdot (-K_X)^2\geq 0$ by \cite{ou}*{Corollary 1.5}.
Since $X$ is rationally connected, we obtain that $h^i(X, \mathcal O_X)=0$ for $i\geq 1$
by \cite{kollar-book}*{Chapter IV.3, Corollary 3.8}.
Then, it follows from \eqref{eq.rr} that
$\chi(X, -mK_X)\geq \chi(X, \mathcal O_X)=h^0(X, \mathcal O_X)=1$ for any $m\geq 0$.
\end{proof}

In the proof of Lemma \ref{lem.rr}, Ou's result \cite{ou}*{Corollary 1.5} 
plays an important role.
This result is a direct corollary of the following theorem, 
which will be crucial in the proof of our non-vanishing result for the case $\nu(X, -K_X)\leq 2$.
This theorem can be essentially reduced to \cite{ou}*{Theorem 1.4},
 since reflexive product preserves semistability on manifolds 
(see \cite{cp11}*{Section 5}). 
For a more general result and a detailed proof, we refer to \cite{lmptx}*{Theorem 4.1}.

\begin{thm}[\cite{ou}*{Theorem 1.4}]\label{thm.pseff}
Let $X$ be a projective manifold and
$(\mathcal T_X)^{\otimes m}\to \mathcal Q$
be a torsion free coherent quotient for some $m\geq 1$. 
If $-K_X$ is nef, then $c_1(\mathcal Q)$ is pseudoeffective.
\end{thm}

\subsection{Results on Campana--Peternell's conjecture}
As stated at the beginning of this paper, Campana--Peternell's conjecture 
can be viewed as a special case of Serrano's conjecture, which 
can be naturally generalized to log pairs
(see \cite{han-liu}*{Question 3.5} 
and \cite{liu-matsumura}*{Conjecture 1.5}). 
For manifolds with strictly nef anti-canonical divisors, 
a special case of the log version of Serrano's conjecture can be stated as follows:

\begin{conj}\label{conj.weak.cp}
Let $X$ be a projective rationally connected manifold such that
the anti-canonical divisor $-K_X$ is strictly nef.
Let $V$ be a nonzero effective divisor on $X$. 
Then, $V-mK_X$ is ample for $m\gg 1$.
\end{conj}

It is obvious that Campana--Peternell's conjecture is equivalent
to Conjecture \ref{conj.weak.cp}
plus the following non-vanishing conjecture for strictly nef anti-canonical divisors 
(Conjecture \ref{conj.cp.eq} $\Leftrightarrow$ Conjecture \ref{conj.weak.cp} $+$ Conjecture \ref{conj.nonvanishing}).

\begin{conj}[Non-vanishing conjecture for strictly nef anti-canonical divisors]\label{conj.nonvanishing}
Let $X$ be a projective rationally connected manifold such that
the anti-canonical divisor $-K_X$ is strictly nef.
Then, we have that $\kappa (X, -K_X)\geq 0$.
\end{conj}

The following lemma is a first step towards Conjecture \ref{conj.weak.cp}.

\begin{lem}\label{lem.firststep}
Let $X$ be a projective rationally connected manifold such that
the anti-canonical divisor $-K_X$ is strictly nef.
Let $V$ be a nonzero effective divisor on $X$. 
Then, there exists a positive integer $m_0$ (dependent on $V$) such that 
$V-mK_X$ is strictly nef  for any integer $m> m_0$.
Moreover, if $V-mK_X$ is big for some $m> m_0+1$,
then $V-mK_X$ is ample. 
\end{lem}

\begin{proof}
For simplicity, 
we can assume that $V$ is prime, and hence $V$ is Cartier on $X$.
Let $n:=\dim X$.
Let $r\in \mathbb Q_{>0}$ be a sufficiently small number such that $(X, r V)$ is a klt pair.
Then, by \cite{liu-matsumura}*{Lemma 2.4}, the $\mathbb Q$-Cartier divisor $K_X+r V+s(-K_X)$ is strictly nef for $s>2n$. 
In particular, $V-mK_X$ is a strictly nef Cartier divisor for $m>m_0:=\lceil \frac{2n-1}{r}\rceil$.

If $V-mK_X$ is big  for some $m> m_0+1$, then 
$(K_X+V-mK_X)-K_X=V-mK_X$ is strictly nef and big.
It follows from the basepoint-free theorem that $K_X+V-mK_X$ is semiample.
That is,  there exists a morphism $f\colon X \to Y$ and a $\mathbb Q$-Cartier ample divisor $H$ on $Y$ such that $K_X+V-mK_X=f^*H$.
Since $K_X+V-mK_X$ is strictly nef, $f$ has to be finite. Hence,
$K_X+V-mK_X=f^*H$ is ample.
In particular, $V-mK_X=(K_X+V-mK_X)+(-K_X)$ is ample.
\end{proof}

The following calculations are much the same as those in \cite{serrano}*{Lemma 1.3} and
\cite{liu-matsumura}*{Theorems 3.3 and 3.5}.

\begin{cor}\label{cor.cal}
Let $X$ be a projective rationally connected 
manifold of dimension $n$ such that
the anti-canonical divisor $-K_X$ is strictly nef.
Let $V$ be a nonzero prime divisor on $X$. 
If $V-mK_X$ is not ample for any $m\gg 1$,
then we have that 
\[
0=V^n=V^{n-1}\cdot K_X=\cdots =V \cdot K_X^{n-1}=K_X^n. 
\]
\end{cor}

\begin{proof}
By Lemma \ref{lem.firststep}, $V-mK_X$ and $V-mK_X-K_X$ are both strictly nef but not big for $m\gg 1$.
In particular, we have $(V-mK_X-K_X)^n=0$ for some $m$.
This yields that 
\[
0=(V-mK_X-K_X)^n=\sum_{k=0}^n \big(_k^n\big)(V-mK_X)^k\cdot (-K_X)^{n-k}.
\]
Then, all the terms on the right-hand side should be zero since $-K_X$ is also nef. 
From $(V-mK_X)\cdot (-K_X)^{n-1}=(-K_X)^n=0$, we have that 
$V\cdot (-K_X)^{n-1}=0$; combining with $(V-mK_X)^2\cdot (-K_X)^{n-2}=0$,
we can conclude that $V^2\cdot (-K_X)^{n-2}=0$, and so forth. Finally,  we can deduce the desired equalities.  
\end{proof}

\subsection{Almost strictly nef divisors}
We recall from \cite{liu-matsumura}*{Subsection 2.1} 
the definition and some useful properties of almost strictly nef divisors,
which are the birational analogues of strictly nef divisors.

\begin{defn}[\cite{ccp}*{Definition 1.1} or \cite{liu-matsumura}*{Definition 2.1}]\label{defn.asn}
A $\mathbb Q$-Cartier divisor $L$ on a normal variety $X$ 
is called {\em{almost strictly nef}}
if there exists a surjective birational 
morphism $\mu\colon X\to X^*$ and a strictly nef $\mathbb Q$-Cartier divisor $L^*$ on $X^*$
such that $L=\mu^*L^*$. 
When a birational morphism $\mu\colon X\to X^*$  is specified, 
we say that $L$ is {\em{almost strictly nef with respect to $\mu$}}. 
\end{defn}

\begin{lem}[\cite{liu-matsumura}*{Lemma 2.4}]\label{lem.almost.nef}
Let $(X, \Delta)$ be a log canonical pair of dimension $n$ and 
$L$ be an almost strictly nef Cartier divisor on $X$
with respect to $\mu\colon X\to X^*$.
If $K_X+\Delta$ is $\mu$-nef,
then the $\mathbb Q$-Cartier divisor $K_{X}+\Delta+tL$ is nef for $t\geq 2n$. 
\end{lem}

For almost strictly nef divisors, 
we have a birational version of the log version of Serrano's conjecture
(see \cite{ccp}*{Conjecture 2.2} or \cite{liu-matsumura}*{Conjecture 2.6}).
For the inductive procedure,
we state the conjecture in the following weakened form:

\begin{conj}\label{conj.big}
Let $X$ be a projective canonical 
variety of dimension $n$ 
and $L$ be an almost strictly nef Cartier divisor on $X$.
Then, the divisor  $K_X+tL$ is big for $t>2n$. 
\end{conj}

\begin{prop}\label{prop.ind}
Conjecture \ref{conj.weak.cp} holds in dimension $n$ if Conjecture \ref{conj.big} holds in dimension $n-1$.
\end{prop}

\begin{proof}
We can assume that $V$ is prime. If $V-mK_X$ is not ample for $m\gg 1$,
then by Corollary \ref{cor.cal}, we have that
\begin{equation}\label{eq.key}
0=V^n=V^{n-1}\cdot K_X=\cdots =V \cdot K_X^{n-1}=K_X^n.
\end{equation}
Let $\mu \colon T\to V$ be
 the composition of the normalization of $V$ and the
 relative canonical model (see \cite{liu-matsumura}*{Subsection 2.3} for the definition) 
of its normalization. As in \cite{liu-matsumura}*{Theorem 3.3},
there exists an effective divisor $B$ such that
$K_T+B=\mu^*(K_X+V)|_T$.
Set 
\[
L:=\mu^*(-K_X)|_T,
\]
which is almost strictly nef with respect to $\mu$.
By the assumption, $K_T+tL$ is big for $t>2n-2$.
It follows that $K_T+B+tL$ is also big for $t>2n-2$. Moreover,
\[
K_T+B+tL=\mu^*(K_X+V-tK_X)|_T
\]
is nef for $t\gg 1$ since $K_X+V-tK_X$ is nef for $t\gg 1$ by Lemma \ref{lem.firststep}.
In conclusion, we obtain that
\[
0<(K_T+B+tL)^{n-1}=(\mu^*(K_X+V-tK_X)|_T)^{n-1}=(K_X+V-tK_X)^{n-1}\cdot V,
\]
contradicting the equation \eqref{eq.key}.
\end{proof}

\begin{rem}
Similar results as Lemma \ref{lem.firststep}, Corollary \ref{cor.cal} and Proposition \ref{prop.ind} also hold
when $K_X$ is strictly nef. 
Combining with \cite{liu-matsumura}*{Theorem 3.5}, we have seen that in 
all the three important special cases of Serrano's conjecture 
($K_X$ is strictly nef, $-K_X$ is strictly nef, and $K_X$ is trivial),
 we can perform the induction on the dimension and reduce the problem 
onto the prime divisor $V$.
\end{rem}


\subsection{Prime Calabi--Yau divisors}
As in \cite{liu-matsumura}, we also have to consider the so-called prime Calabi--Yau divisors
in this paper. 
A prime divisor $D$ on a normal variety $X$ is said to be a \emph{prime Calabi--Yau divisor}  if
$D$ is a normal variety with at worst canonical singularities 
satisfying that $K_D\sim 0$ and  $H^1(D, \mathcal{O}_D)=0$. 

\begin{lem}\label{lem.nef}
Let $X$ be a projective manifold such that
the anti-canonical divisor $-K_X$ is nef.
Let  $D$ be a prime Calabi--Yau divisor on $X$.
Then,  $D$ is nef.
\end{lem}

\begin{proof}
By the adjunction formula \cite{kollar-mori}*{Proposition 5.73},
we have that $(K_X+D)|_D=K_D\sim 0$.
Let $C$ be a curve on $X$. If $C\not\subset D$, then we have $D\cdot C\geq 0$. 
If $C\subset D$, then we have $D\cdot C=D|_D\cdot C=(K_D-K_X|_D)\cdot C=-K_X\cdot C \geq 0$.
Therefore, $D$ is nef.
\end{proof}

\begin{rem}
We assume that $X$ is a projective manifold of dimension $n$ such that
$c^{n-2}_1(X)\cdot c_2(X)=0$ and the anti-canonical divisor $-K_X$ is nef.
We also assume that $D\sim-K_X$ is a smooth prime Calabi--Yau divisor.
Then, from the exact sequence
\[
0\to \mathcal{T}_D \to \mathcal {T}_X|_D \to \mathcal O_D(D)\to 0,
\]
we have that $c_2(X)|_D=c_2(D)+c_1(D)\cdot D|_D=c_2(D)$. 
It follows that 
\[
(-K_X|_D)^{n-3}\cdot c_2(D)=(-K_X|_D)^{n-3} \cdot c_2(X)|_D=(-K_X)^{n-2}\cdot c_2(X) =
c^{n-2}_1(X)\cdot c_2(X)=0.
\]
In particular, we have that $(-K_X|_D) \cdot c_2(D)=0$ in dimension $n=4$.
If moreover $-K_X$ is strictly nef, then these conditions turn out to be the remaining case
of Serrano's conjecture in dimension 3.
This is the only reason that we have to treat Campana--Peternell's conjecture in dimension 4
under the assumption that  $c^2_1(X)\cdot c_2(X)\neq 0$.   
\end{rem}

\section{Partial results on Conjecture \ref{conj.weak.cp}}\label{sec3}

In this section, we prove Conjecture \ref{conj.weak.cp} in dimension 4
under the assumption that $c^2_1(X)\cdot c_2(X)\neq 0$.

\begin{thm}\label{thm.lv}
Let $X$ be a projective rationally connected smooth fourfold such that 
the anti-canonical divisor $-K_X$ is strictly nef.
Let $V$ be a nonzero prime divisor on $X$. 
If one of the following conditions 
\begin{enumerate}
 \item $V$ is not a prime Calabi--Yau divisor,
 \item $V\not\sim -K_X$, \text{or}
 \item $c^2_1(X)\cdot c_2(X)\neq 0$
\end{enumerate}
holds, then $V-mK_X$ is ample for $m\gg 1$.
\end{thm}

\begin{proof} 
The first two steps of the proof  follow exactly 
the proofs of \cite{liu-matsumura}*{Theorems 3.3 and 3.5}. 
We sketch the proof here for the reader's convenience.

\begin{step} 
We assume that $V-mK_X$ is not ample for $m\gg 1$.
By Corollary \ref{cor.cal}, we obtain that
\begin{equation}\label{eq.key.1}
0=V^4=V^3 \cdot K_X=V^2 \cdot K_X^2=V \cdot K_X^3=K_X^4.
\end{equation}
Let $T$ be a relative canonical model of the normalization of $V$ and $\mu\colon T\to V$ 
be the corresponding morphism. 
As in \cite{liu-matsumura}*{Theorem 3.3},
there exists an effective divisor $B$ such that
$K_T+B=\mu^*(K_X+V)|_T$. Set $L:=\mu^*(-K_X)|_T$, which is almost strictly nef with respect to $\mu$.
Then for $t\gg 1$, $K_T+tL$ is nef by Lemma \ref{lem.almost.nef},
but not big by Proposition \ref{prop.ind}.
Similarly, set 
\[D_m:=\mu^*(K_X+V-mK_X)|_T,
\]
which is almost strictly nef with respect to $\mu$ 
for any $m\gg 1$ since $K_X+V-mK_X$ is strictly nef for any $m\gg 1$ 
by Lemma \ref{lem.firststep}.
Then, $K_T+6D_m$ is nef for $m\gg 1$ by Lemma \ref{lem.almost.nef}.
For $m\gg 1$, we can see that
\[
K_T+B+6D_m=\mu^*(7(K_X+V)-6mK_X)|_T=\mu^*(7V-(6m-7)K_X)|_T
\]
is nef by Lemma \ref{lem.firststep},  but not big by \eqref{eq.key.1}.
It follows that $K_T+6D_m$ is nef but not big for $m\gg 1$. 
Then using similar calculations for $K_T+tL$ and $K_T+6D_m$  as in Corollary \ref{cor.cal},
we have that 
\begin{equation}\label{eq.key.2}
\begin{split}
&K_T^3=K_T^2 \cdot L=K_T \cdot L^2=L^3=0,
\\ 
&K_T^3=K_T^2 \cdot D_m=K_T \cdot D_m^2=D_m^3=0,
\\ 
&K_T^3=K_T^2 \cdot D_{m+1}=K_T \cdot D_{m+1}^2=D_{m+1}^3=0,
\\ 
&2K_T \cdot D_{m}\cdot L=K_T \cdot (D_{m+1}^2-D_m^2-L^2)=0
\end{split}
\end{equation}
for any $m\gg 1$.
Directly from \eqref{eq.key.1} or from $(D_m+L)^3=D_{m+1}^3=0$ in \eqref{eq.key.2}, 
we can see that $D_m^2\cdot L=D_m \cdot L^2=0$.
It follows that  
\begin{equation}\label{eq.key.3}
\begin{split}
&L^2 \cdot B=L^2 \cdot  (K_T+B)= L^2 \cdot(D_m-mL)=0,
\\ 
&D_{m}^2 \cdot B=D_m^2 \cdot  (K_T+B)=D_m^2 \cdot  (D_m-mL)=0,
\\ 
&L \cdot K_T \cdot  B=L \cdot  K_T \cdot  (K_T+B)=L \cdot  K_T \cdot  (D_m-mL)=0,
\\
&D_{m} \cdot K_T \cdot  B=D_m \cdot  K_T \cdot  (K_T+B)=D_m \cdot  K_T \cdot  (D_m-mL)=0,
\end{split}
\end{equation}
and
\begin{equation}\label{eq.key.4}
\begin{split}
&K^2_T \cdot  B=K_T^2 \cdot  (K_T+B)=K_T^2 \cdot (D_m-mL)=0,
\\ 
&K_T \cdot  B^2=K_T \cdot  (K_T+B)^2=K_T \cdot (D_m-mL)^2=0,
\\ 
&B^3=(K_T+B)^3=(D_m-mL)^3=0
\end{split}
\end{equation}
for any $m\gg 1$. In the next step, we derive contradictions on $B$ using these equations as in 
\cite{liu-matsumura}*{Theorems 3.3 and 3.5}.
\end{step}

\begin{step}
If some integral component $S$ of $B$ is not $\mu$-exceptional, then
as $L$ and $K_T+tL$ are nef, we have that $0=L^2 \cdot  B\geq a_S L^2 \cdot  S\geq 0$ and
\[
0=(K_T+tL) \cdot L \cdot B\geq a_S(K_T+tL) \cdot L \cdot S\geq 0
\] 
by \eqref{eq.key.3}, where $a_S> 0$ is the coefficient of $S$ in $B$. 
Similar results also hold after replacing $L$ by $D_m$ and $K_T+tL$ by $K_T+6D_m$ for $m\gg 1$.
These indicates that 
\begin{equation}\label{eq.key.5}
\begin{split}
&L \cdot K_T \cdot S=L^2 \cdot S=0,
\\ 
&D_m \cdot K_T \cdot S=D_m^2 \cdot S=0,
\\ 
&2L\cdot D_m \cdot S=(D_{m+1}^2-D_m^2-L^2) \cdot S=0
\end{split}
\end{equation}
for $m\gg 1$.
By exactly the same arguments as in Case 1 of the proof of \cite{liu-matsumura}*{Theorems 3.3},
there exists a curve $C\in |(K_T+S+rL)|_S|$ which is not contained in the $\mu|_S$-exceptional locus
for $r>3$. 
Then, from \eqref{eq.key.5}, we can conclude that
\[
0<L|_S\cdot C=L \cdot (K_T+S+rL)  \cdot S=L \cdot K_T  \cdot S+L \cdot S^2
+rL^2 \cdot S=L \cdot S^2.
\]
On the other hand, from \eqref{eq.key.5},
we have that 
\[
\begin{split}
0&=L \cdot (D_m-mL) \cdot S=L \cdot (K_T+B)  \cdot S=L \cdot K_T \cdot S+L\cdot B\cdot S
\\ 
&=L\cdot (B-a_SS) \cdot S+a_S L\cdot S^2\geq a_S L\cdot S^2. 
\end{split}
\]
This is a contradiction.

Therefore, we assume that $\dim \mu(B)\leq 1$.
In this case, $B$ is $\mu$-exceptional and $V$ is normal. 
Assume that $B\neq 0$. Then, $\dim \mu(B)=1$ by $B^3=0$ in \eqref{eq.key.4}.
As in \cite{liu-matsumura}*{Theorem 3.3}, 
we have that $(K_T+\mu^*H) \cdot B \cdot L>0$ for a very ample divisor $H$ on $V$. 
It follows that
 $H \cdot \mu(B) \cdot (-K_X)|_V>0$, which is impossible. Therefore, we have that $B=0$, 
$T=V$ is a Gorenstein variety with at worst canonical singularities, 
$L=-K_X|_V$ is strictly nef,
and $K_V+tL$ is not ample for any $t\gg 1$ (see Proposition \ref{prop.ind}).
The same as the last paragraph of the proof of  \cite{liu-matsumura}*{Theorem 3.5},
we can assume that $q(V)=h^1(V,\mathcal O_V)=0$ and $K_V\sim_{\mathbb Q}0$.
By \cite{liu}*{Lemma 2.1} (whose proof is based on the strategy of \cite{lp}),
we can further assume that $K_V\sim 0$, that is, $V$ is a prime Calabi--Yau divisor. In conclusion of the first two steps, we prove the desired result under condition (1):
\begin{claim}
If $V$ is not a prime Calabi--Yau divisor, then $V-mK_X$ is ample for $m\gg 1$.
\end{claim}
\end{step}

\begin{step}
In this step, we consider that $V$ is a prime Calabi--Yau divisor. In particular,
$0\sim K_V=(K_X+V)|_V$ by definition and the adjunction formula \cite{kollar-mori}*{Proposition 5.73}.
Then, $V$ is nef on $X$ by Lemma \ref{lem.nef}.
We consider the long cohomology sequence induced by the exact sequence
\begin{equation}\label{eq.pcy}
0\to \mathcal O_X(K_X)\to \mathcal O_X(K_X+V)\to \mathcal O_V(K_X+V)\simeq \mathcal O_V\to 0.
\end{equation}
Since $X$ is rationally connected,
we have that  $H^0(X, \mathcal O_X(K_X))=H^4(X, \mathcal O_X)=0$ and
$H^1(X, \mathcal O_X(K_X))=H^3(X, \mathcal O_X)=0$ by Serre's duality. 
It follows that 
\[
H^0(X, \mathcal O_X(K_X+V))\cong H^0(V, \mathcal O_V)\cong \mathbb C.
\]
That is, there exists an effective divisor $V_1\sim K_X+V$ and 
all the coefficients of integral components of $V_1$ are positive integers. 
If $V_1 \neq 0$ and some integral component $E$ of $V_1$ is not a prime Calabi--Yau divisor,
then $E-mK_X$ is ample for $m\gg 1$ by the above Claim; 
it follows that 
\[
V-(2m-1)K_X=V_1-2mK_X=E-mK_X+(V_1-E)-mK_X
\]
is also ample for $m\gg 1$, since $(V_1-E)-mK_X$ is strictly nef for $m\gg 1$ by Lemma \ref{lem.firststep}.
Therefore, we assume that any integral component of $V_1$ is a prime Calabi--Yau divisor.
Note that in this case, $V_1$ is nef by Lemma \ref{lem.nef} again.

Repeatedly using \eqref{eq.pcy} for each integral component of $V_{i-1}$, 
we will get a sequence of $V_i$  which is 
a union of prime Calabi--Yau divisors  
such that $V_{i}\sim r_iK_X+V_{i-1}$ for some integer $r_{i} \geq 1$
and $V-(\sum r_i)(-K_X)=V_i$ is nef by Lemma \ref{lem.nef}.
If $V_i\neq 0$ for infinitely many $i$, then for a fixed curve $C$ on $X$,
we will have 
\[
V_i\cdot C=V\cdot C-(\sum r_i)(-K_X)\cdot C<0
\]
for $i\gg 1$ since $-K_X$ is strictly nef and $\sum r_i\to +\infty$ for $i\to +\infty$. 
This contradicts that $V_i$ is nef. 
Hence, we must stop after finitely 
many steps, that is, there exists some integer $n>0$ such that $V_n=0$. It follows that $V\sim -(\sum_{i=1}^n r_i)K_X$. 
Since $K_V=(K_X+V)|_V\sim 0$ and $-K_X$ is strictly nef, we must have that $\sum_{i=1}^n r_i=1$.
That is, $V_1=0$ and $V\sim -K_X$.
In conclusion, we prove in this step that if $V\not\sim -K_X$, then $V-mK_X$ is ample for $m\gg 1$.
\end{step}
\begin{step}
Finally, in the case  $V\sim -K_X$,  we consider the exact sequence:
\[
0\to \mathcal O_X(-mK_X+K_X)\to \mathcal O_X(-mK_X)\to \mathcal O_V(-mK_X)\simeq \mathcal O_V(mV)\to 0.
\]
By \eqref{eq.rr} and condition $(3)$, we obtain that 
\[
\chi(V,mV)=\chi(X, -mK_X)-\chi(X, -mK_X+K_X)
=\frac{c^2_1(X)\cdot c_2(X)}{12}m>0
\]
for $m\geq 1$. In this case, it is well-known that the strictly nef divisor 
$V|_V=-K_X|_V$ is ample (for example, see the proof of \cite{liu}*{Lemma 2.1} and the references therein).
In particular, $(-K_X)^4=(-K_X)^3\cdot V=(-K_X|_V)^3>0$, which implies that $-K_X$ is ample on $X$. \qedhere
\end{step}
\setcounter{step}{0}
\end{proof}

The ampleness part of Theorem \ref{thm.main} follows from Theorem \ref{thm.lv}:

\begin{cor}\label{cor.main}
Let $X$ be a smooth projective  rationally connected  fourfold such that $c^2_1(X)\cdot c_2(X)\neq 0$ 
and the anti-canonical divisor $-K_X$ is strictly nef.
If $\kappa(X, -K_X)\geq 0$, then $-K_X$ is ample. 
\end{cor}

\begin{proof}
Since $\kappa(X, -K_X)\geq 0$, there exists a positive integer $m$ such that
$-mK_X=\sum a_iV_i$, where $a_i>0$ and $V_i$ is prime for each $i$.
Applying Theorem \ref{thm.lv} to each $V_i$, we obtain that 
$V_i-m_iK_X$ is ample for $m_i\gg 1$. 
Therefore, 
\[
-mK_X-\sum a_im_iK_X=\sum a_i(V_i-m_iK_X)
\]
is ample for $m_i\gg 1$ and all $i$. It follows that $-K_X$ is ample.
\end{proof}

The next corollary of Theorem \ref{thm.lv} is very similar to \cite{liu-matsumura}*{Lemma 4.5}.
It plays the same role in Section \ref{sec4} as \cite{liu-matsumura}*{Lemma 4.5} in 
\cite{liu-matsumura}*{Section 4}.

\begin{cor}\label{cor.pseff}
Let $X$ be a smooth projective  rationally connected  fourfold such that
the anti-canonical divisor $-K_X$ is strictly nef.
If $\kappa(X, -K_X)=-\infty$,  
then the divisor $-K_X-D$ is not pseudoeffective for any nonzero effective $\mathbb R$-divisor $D$ on $X$.
\end{cor}

\begin{proof}
Assume that there exists a nonzero effective $\mathbb R$-divisor $D$ 
such that $E:=-K_X-D$ is  pseudoeffective. 
Since $\kappa(X, -K_X)=-\infty$, $-K_X$ cannot be linearly equivalent to any prime Calabi--Yau divisor.
Then by Theorem \ref{thm.lv} and 
the same proof of Corollary \ref{cor.main}, 
the divisor $H:=D-mK_X$ is ample for $m\gg 1$. 
Since $E$ is pseudoeffective and $H$ is ample, 
we have that 
\begin{equation*}
\begin{split}
&(-K_X)^3\cdot E\geq 0, \quad (-K_X)^2\cdot  H \cdot E\geq 0, \quad  (-K_X)\cdot  H^2 \cdot  E\geq 0,
\\ 
&(-K_X)^3 \cdot H\geq 0, \quad  (-K_X)^2 \cdot  H^2\geq 0, \quad  (-K_X)\cdot  H^3\geq 0.
\end{split}
\end{equation*}
Since $\kappa(X,-K_X)=-\infty$, we obtain that $(-K_X)^4=0$.
From $0=(m+1) (-K_X)^4=(-K_X)^3 \cdot (E+H)\geq 0$, we obtain $(-K_X)^3 \cdot E=(-K_X)^3 \cdot H=0$. 
Then, we obtain $0=(m+1)(-K_X)^3 \cdot H=(-K_X)^2\cdot(E+H)\cdot H\geq 0$,  
which implies that $(-K_X)^2\cdot H \cdot E=(-K_X)^2 \cdot H^2=0$.
Similarly, 
from $$
0=(m+1)(-K_X)^2 \cdot H^2=(-K_X)\cdot (E+H) \cdot H^2=(-K_X)\cdot H^2 \cdot E+(-K_X)\cdot H^3\geq 0,
$$
we finally obtain $(-K_X)\cdot H^2 \cdot E=(-K_X) \cdot H^3=0$.
However, a multiple of the numerical class $H^3$
can be represented by an effective curve since $H$ is ample, 
and thus $(-K_X) \cdot H^3=0$ contradicts the strict nefness of $-K_X$. 
\end{proof}

\section{Non-vanishing for strictly nef anti-canonical divisors}\label{sec4}

In this section, we prove the non-vanishing part of Theorem \ref{thm.main}.

\begin{thm}[Theorems \ref{thm.nu=3} and \ref{thm.nu=2}]\label{thm.nonvanishing}
Let $X$ be a projective rationally connected smooth fourfold such that
the anti-canonical divisor $-K_X$ is strictly nef.
Then, we have that $\kappa(X, -K_X)\geq 0$.
\end{thm}

If $\nu(X, -K_X)=4$, then the anti-canonical divisor $-K_X$ is big.
By the basepoint-free theorem, $-K_X$ is semiample. That is,
there exists a morphism $f\colon X \to Y$ and a $\mathbb Q$-Cartier ample divisor $H$ on $Y$ such that $-K_X=f^*H$.
Since $-K_X$ is strictly nef, $f$ has to be finite. Hence,
$-K_X=f^*H$ is ample.
If $\nu(X, -K_X)=0$, then $-K_X\cdot D^3=0$ for any very ample divisor $D$ on $X$ by definition. 
However, a multiple of the numerical class $D^3$ is represented by an effective curve,
hence $-K_X\cdot D^3>0$ by the strict nefness of $-K_X$, which is impossible.
Therefore, it suffices to consider $\nu(X, -K_X)=1, 2, 3$. 

\subsection{The case $\nu(X, -K_X)=3$}
In this case, we have the non-vanishing result for nef anti-canonical divisors by the standard arguments,
see \cite{lp17}*{Lemma 2.1} and the references therein. We sketch the proof for the reader's convenience.

\begin{thm}\label{thm.nu=3}
Let $X$ be a projective rationally connected smooth fourfold such that
the anti-canonical divisor $-K_X$ is nef.
If $\nu(X,-K_X)=3$, then $\kappa(X,-K_X)\geq 0$.
\end{thm}

\begin{proof}
Let $H$ be a very ample smooth hypersurface on $X$. Then, $-K_X|_H$ is nef and big on $H$ 
(see \cite{liu-svaldi}*{Lemma 2.1}). 
Let us consider the long cohomology sequence  induced by the exact sequence:
\[
0\to \mathcal O_X(-mK_X) \to  \mathcal O_X(H-mK_X)\to \mathcal O_H(H-mK_X)\to 0.
\]
By the adjunction formula, we have that 
\[
(H-mK_X)|_H=(K_X+H-(m+1)K_X)|_H=K_H+(m+1)(-K_X)|_H.
\] 
Therefore, it follows from 
the Kawamata--Viehweg vanishing theorem that
\[
H^i(H, \mathcal O_H(H-mK_X))=H^i(H, \mathcal O_H(K_H+(m+1)(-K_X)|_H))=0
\]
for $i\geq 1$ and $m\geq 0$. On the other hand, since $H+m(-K_X)$ is ample for any $m\geq 0$,
we have 
\[
H^i(X, \mathcal O_X(H-mK_X))=H^i(X, \mathcal O_X(K_X+H+(m+1)(-K_X)))=0
\] 
for any $i\geq 1$ and $m\geq 0$ by Kodaira's vanishing theorem.
Therefore, we obtain that 
$H^i(X, \mathcal O_X(-mK_X))=0$ holds for $i\geq 2$ and  $m\geq 0$.
Note that  $\chi(X, -mK_X)\geq 1$ for $m\geq 0$ by Lemma \ref{lem.rr}.
Hence, 
\[
h^0(X, \mathcal O_X(-mK_X))=h^1(X, \mathcal O_X(-mK_X))+\chi(X, -mK_X)\geq 1,
\]
which implies that $\kappa(X,-K_X)\geq 0$.
\end{proof}

\subsection{The case $\nu(X, -K_X)\leq 2$} In this case, we follow closely the arguments in \cite{liu-matsumura}*{Subsection 4.2}, which is based on the strategy of \cites{lop, lp} and the references therein.

\begin{lem}\label{lem.van.1}
Let $X$ be a projective rationally connected smooth fourfold such that
the anti-canonical divisor $-K_X$ is strictly nef.
If  $\kappa(X,-K_X)=-\infty$, 
then $H^0(X, \Omega^q_X(-mK_X))=0$ for $0< q< n$ and $m\gg 1$.
\end{lem}
\begin{proof}
The proof is very similar to \cite{lop}*{Proposition 3.4}.
If there exists an integer $q$ such that 
$H^0(X, \Omega^q_X(-mK_X))\neq 0$ for infinitely many positive integers $m$,
then as in \cite{lop}*{Proposition 3.4} or \cite{lp}*{Lemma 4.1},
there exists a positive integer $r$ and a Cartier divisor $F$ such that
$\mathcal O_X(-F)$ is a subsheaf saturated in $\bigwedge^r \Omega^q_X$ and
\[
H^0(X, \mathcal O_X(-F-m'K_X))\neq 0
\]
for infinitely many positive integers $m'$. 
Therefore, there exists an effective divisor $E_{m'}$ such that
\[
-m'K_X\sim E_{m'}+F
\] for infinitely many positive integers $m'$.
By the dual of the quotient in Theorem \ref{thm.pseff}, 
we have that $F$ is pseudoeffective.
Then, Corollary \ref{cor.pseff} implies that $E_{m'}=0$ for all $m'$.
In particular, $-m'_1K_X\sim F\sim -m'_2K_X$ for two different integers $m'_1$ and $m'_2$,
which implies that $K_X$ is trivial.
This contradicts both that $\kappa(X, -K_X)=-\infty$ and that $-K_X$ is strictly nef.
\end{proof}

\begin{rem}\label{rem.analytic}
In the rest of this paper, 
we use the notation of the analytic methods in \cite{demailly-book} 
and interchangeably use the terms ``Cartier divisors'', ``invertible sheaves'', and ``line bundles''. 
For the compatibility of intersection of positive currents 
(if it is well-defined in the sense of \cite{dem91}*{Corollary 3.3 or Theorem 3.5}) 
and intersection of the corresponding classes in coholomogy, see \cite{dem91}*{Corollary 10.2} (see also \cite{boucksom}*{Subsection 2.6} for a nice survey).  
In particular, 
the intersecting of the class of a positive $(p,p)$-current and $n-p$ nef 
classes (viewing as limits of classes of K\"ahler forms) in coholomogy is non-negative.
\end{rem}

\begin{lem}\label{lem.partial.amp}
Let $X$ be a projective rationally connected smooth fourfold such that 
the anti-canonical divisor $-K_X$ is strictly nef.
Assume that $-K_X$ admits a singular Hermitian metric $h$ with positive curvature current 
such that the closed subschemes $V_m$ defined by multiplier ideal sheaves  $\mathcal I(h^{\otimes m})$ 
are of dimension $ \leq 1$ for any  $m \gg 1$. 
Then, we have  $\kappa(X,-K_X)\geq 0$. 
\end{lem}

\begin{proof}
The proof is based on the strategy of \cite{lop}*{Section 3}.
For a contradiction, we assume that $\kappa(X, -K_X)=-\infty$. 
By the hard Lefschetz theorem \cite{dps}*{Theorem 0.1},
the morphism
\[
H^0(X, \Omega^2_X(-mK_X)\otimes \mathcal I(h^{\otimes m}))
\twoheadrightarrow H^2(X, \strutt_X(K_X-mK_X)\otimes \mathcal I(h^{\otimes m}))
\]
is surjective. 
By Lemma \ref{lem.van.1}, we have that $H^0(X, \Omega^2_X(-mK_X))=0$ for any $m \gg 1$.
It follows that $H^0(X, \Omega^2_X(-mK_X)\otimes \mathcal I(h^{\otimes m}))
=H^2(X, \strutt_X(K_X-mK_X)\otimes \mathcal I(h^{\otimes m}))=0$ 
for any $m \gg 1$. 
On the other hand, we have that
$H^2(V_m, K_X-mK_X)=0$ for any $m \gg 1$ by the dimensional reason. 
Then, from the long cohomology sequence induced by the exact sequence
\[
0\to \strutt_X(K_X-mK_X)\otimes \mathcal I(h^{\otimes m}) \to  \strutt_X(K_X-mK_X) \to \strutt_{V_m}(K_X-mK_X)\to 0,
\]
we can see that $H^2(X,  \strutt_X(K_X-mK_X))=0$ for any $m \gg 1$.
Note that by Serre's duality, we have that $H^4(X,  \strutt_X(K_X-mK_X))=H^0(X,  \strutt_X(mK_X))=0$.
Note also that $\chi(X, K_X-mK_X)\geq 1$ for $m\geq 1$ by Lemma \ref{lem.rr}.
This indicates that 
\[
h^0(X, \mathcal O_X(-(m-1)K_X))\geq \chi(X, K_X-mK_X)\geq 1
\]
for any $m \gg 1$, contradicting the assumption $\kappa(X,-K_X)=-\infty$.
\end{proof}

Then, the proof of \cite{liu-matsumura}*{Theorem 4.7} works almost verbatim 
after replacing \cite{liu-matsumura}*{Lemma 4.5} by Corollary \ref{cor.pseff}
and \cite{liu-matsumura}*{Lemma 4.6} by Lemma \ref{lem.partial.amp}.
We sketch the proof here for the reader's convenience.

\begin{thm}\label{thm.nu=2}
Let $X$ be a projective rationally connected smooth fourfold such that  
 the anti-canonical divisor $-K_X$ is strictly nef.
If $\nu(X,-K_X)\leq 2$, then $\kappa(X,-K_X)\geq 0$.
\end{thm}

\begin{proof}
For a contradiction, we assume that $\kappa(X, -K_X)=-\infty$. 
Let $h$ be a singular Hermitian metric on $-K_X$
such that the curvature current $\sqrt{-1}\Theta_{h}(-K_X)$ is positive. 
Consider the Siu decomposition of $\sqrt{-1}\Theta_{h}(-K_X)$: 
\[
\sqrt{-1}\Theta_{h}(-K_X) = R + \sum_i\lambda_i[D_i],
\]
where  $[D_i]$ is a current of integration over a prime divisor $D_i$
and the positive real number $\lambda_i$ is the generic Lelong number along $D_i$.
If $[D_i]$ is nonzero for some $i$, then
$-mK_X - D_{i}$ is pseudoeffective 
for $m \gg 1$,
contradicting Corollary \ref{cor.pseff}. 

Therefore, we assume that $\sum_i\lambda_i[D_i]=0$. In this case,
as in the proof of \cite{liu-matsumura}*{Theorem 4.7},
we can use the Bedford--Taylor product \cite{dem}*{Lemma 7.4} and
the approximation theorem \cite{dem}*{Main Theorem 1.1} 
to get a  Siu decomposition of $(-K_X)^{2}$:
\begin{align}\label{lelong}
(-K_X)^{2} \ni T + S \geq T  + \sum_{S_{i}}\nu(R,S_i)^{2} [S_{i}],
\end{align}
where $T, S$ are positive $(2,2)$-currents,
$\nu(R, S_i)$ is the Lelong number along
any subvariety $S_i \subset X$ of codimension 2, 
and the inequality in \eqref{lelong} follows from \cite{dem}*{(7.5)} and the properties of the Bedford--Taylor product.
Taking the wedge products of a K\"ahler form  $\{\omega\}$ and the numerical class $-K_X$ yields 
that 
$$
0=(-K_X)^{3} \cdot \{\omega\} 
\geq  \{T\} \cdot (-K_X)  \cdot  \{\omega\} + \{\sum_{S_{i}}\nu(R,S_i)^{2} [S_{i}] \}\cdot (-K_X)  
\cdot  \{\omega\},   
$$
where $\{\bullet \}$ denotes the the numerical class of $(p,p)$-currents. 
The equality on the left-hand side follows from the assumption $\nu(X, -K_X)\leq 2$. 
Then, since $-K_X$ is nef and $T$ is a positive $(2,2)$-current, 
each non-negative term (see Remark \ref{rem.analytic}) on the right-hand side  is zero. 
 In particular, we have that $(-K_X)|_{S_i}\equiv 0$ 
for any subvariety $S_i$ with $\nu(R, S_i) >0$. 
On the other hand, we can choose a curve $C$ on $S_i$ such that $-K_X\cdot C >0$ since $-K_X$ is strictly nef.
This is a contradiction,
and so the upper-level set of Lelong numbers 
\[
E_{c}(\sqrt{-1}\Theta_h(-K_X)):=\{x \in X \,|\, \nu(\sqrt{-1}\Theta_h(-K_X), x) \geq c\}
\] is a (possibly) 
countable union of Zariski closed subsets of codimension $\geq 3$ for any $c>0$. 
Then, by Skoda's lemma (see \cite{demailly-book}*{(5.6) Lemma} for example), we obtain that
\[
V_m\subset E_1(\mathcal I(h^{\otimes m}))=E_{\frac{1}{m}}(\sqrt{-1}\Theta_h(-K_X)),
\]
where $V_m$ is defined as in Lemma \ref{lem.partial.amp}.
It follows that $V_m$ is of dimension $\leq 1$ for any $m\geq 1$, 
and hence $\kappa(X,-K_X) \geq 0$ by Lemma \ref{lem.partial.amp},
which is a contradiction.
\end{proof}

\begin{rem}
In \cite{kollar-conj}*{Remark 3.6},
 Koll\'{a}r expected to give a structure theory for minimal threefolds
with $c_1\cdot c_2=0$,
 and find a universal lower bound $-c_1\cdot c_2\geq \epsilon >0$ for the case with $c_1\cdot c_2\neq 0$.
This expectation can be naturally generalized to any dimensions. 
Inspired by the works \cites{hs, ou} and the remaining case of 
Campana--Peternell's conjecture in dimension 4,
we propose the following natural question:
\begin{ques}\label{ques.main}
For projective  terminal varieties of dimension $n$ with nef anti-canonical divisors,
is there a structure theory for the case $c^{n-2}_1 \cdot c_2 =0$ and a universal lower bound $\epsilon$
depending only on $n$ such that $c^{n-2}_1 \cdot  c_2 \geq \epsilon >0$ for the case with $c_1\cdot c_2\neq 0$?
\end{ques}

This question is simple in the surface case; in dimension 3, the author and Chen Jiang 
provide a partial answer to this question in a recent note \cite{jiang-liu}.
We hope that a satisfactory structure theory for $c^2_1(X)\cdot  c_2(X)=0$ in dimension 4
can provide the last piece of the puzzle on Campana--Peternell's conjecture in dimension 4.
\end{rem}


\end{document}